\documentclass[11pt]{amsart}
\usepackage{amsmath, amsfonts, amsthm, amssymb, multicol, mathtools, dsfont, verbatim}
\usepackage[shortlabels]{enumitem}
\usepackage[shortlabels]{enumitem}
\usepackage{graphicx}
\usepackage{tikz-cd}
\usepackage{lipsum}
\usepackage{mathtools}
\usepackage{float, hyperref}
\usepackage{scalerel,stackengine, subcaption}
\usepackage[shortlabels]{enumitem}
\usepackage{yfonts}
\usepackage[margin=1in]{geometry}
\usepackage{pgfplots}
\usepackage{tikz}
\usepgfplotslibrary{fillbetween}
\pgfplotsset{width=10cm,compat=1.9}
\usepackage[square,sort,comma,numbers]{natbib}
\allowdisplaybreaks

\usepackage[square,sort,comma,numbers]{natbib}
\DeclarePairedDelimiterX{\bignorm}[1]{\bigg\Vert}{\bigg\Vert}{#1}
\DeclarePairedDelimiterX{\norm}[1]{\lVert}{\rVert}{#1}
\usepackage{fancyhdr}

\makeatletter
\def\@setauthors{%
  \begingroup
  \def\thanks{\protect\thanks@warning}%
  \trivlist
  \centering\footnotesize \@topsep30\p@\relax
  \advance\@topsep by -\baselineskip
  \item\relax
  \author@andify\authors
  \def\\{\protect\linebreak}

  \normalsize\lowercase{\authors}%
  
	\ifx\@empty\contribs
  \else
    ,\penalty-3 \space \@setcontribs
    \@closetoccontribs
  \fi
  \endtrivlist
  \endgroup
}
\def\@settitle{\begin{center}
\LARGE\lowercase{\@title}
  \end{center}%
}
\makeatother

\definecolor{lightblue}{HTML}{2B77A4}
\definecolor{darkred}{HTML}{9E0D0D}
\hypersetup{
	colorlinks=true,
	linkcolor=darkred,
	urlcolor=darkred,
	citecolor=lightblue
}
\numberwithin{equation}{section}

\newtheorem{thm}{Theorem}[section]
\newtheorem{lma}[thm]{Lemma}
\newtheorem{cor}[thm]{Corollary}
\newtheorem{defn}[thm]{Definition}

\newtheorem{prop}[thm]{Proposition}
\newtheorem{conj}[thm]{Conjecture}

\newtheorem*{propA}{Proposition A}
\renewcommand{\epsilon}{\varepsilon}

\theoremstyle{definition}
\newtheorem{definition}{Definition}[section]
\theoremstyle{conjecture}

\newcommand{\pare}[1]{\left( #1 \right)}
\newcommand{\paree}[1]{\left\{ #1 \right\}}

\newcommand{\mm}[1]{\[ #1 \]}
\newcommand{\abso}[1]{\left| #1 \right|}
\newcommand{\nn}[1]{\norm{ #1 }}
\newcommand{\intt}[2]{\int_{#1}^{#2}}

\newcommand{\bigcupp}[2]{\bigcup_{#1}^{#2}}
\newcommand{\nnn}[3]{\left\Vert #1 \right\Vert_{#2}^{#3}}
\newcommand{\bignnn}[3]{\bignorm{ #1 }_{#2}^{#3}}
\newcommand{\summ}[2]{\sum_{#1}^{#2}}
\newcommand{\RR}{\mathbb{R}}
\newcommand{\TT}{\mathbb{T}}
\newcommand{\cK}{\mathcal{K}}
\newcommand{\cM}{\mathcal{M}}
\newcommand{\BB}{\mathcal{B}}
\newcommand{\EE}{\mathcal{E}}
\newcommand{\FF}{\mathcal{F}}
\newcommand{\McK}[1]{\mathcal{K}_{\delta,A}(#1)}
\newcommand{\tube}[2]{T^{\delta}_{#1}(#2)}
\newcommand{\ubox}{\overline{\textup{dim}}_\textup{B}}
\newcommand{\lbox}{\underline{\textup{dim}}_\textup{B}}

\newcommand{\pacd}{\textup{dim}_\textup{P}}

\newcommand{\haus}{{\textup{dim}}_\textup{H}}

\renewcommand{\geq}{\geqslant}
\renewcommand{\leq}{\leqslant}

\newcommand{\ubd}{\overline{\dim}_{\textup{B}}}

\newcommand{\pd}{\dim_{\textup{P}}}

\title{Hausdorff dimension of restricted Kakeya sets}

\author{ Jonathan M. Fraser${}^{1}$ and Lijian Yang${}^{2}$\\ \\
 ${}^{1,2}$ School of Mathematics and Statistics, University of St Andrews, Scotland\\
\MakeLowercase{Emails: ${}^{1}$jmf32@st-andrews.ac.uk and ${}^{2}$ly51@st-andrews.ac.uk}}

\begin{document}


\maketitle
\thispagestyle{empty}

\begin{abstract}
A Kakeya set in $\RR^n$ is a compact set that contains a unit line segment $I_e$ in each direction $e \in S^{n-1}$. The Kakeya conjecture states that any Kakeya set in $\RR^n$ has Hausdorff dimension $n$. We consider a restricted case where the  midpoint of each line segment $I_e$ must belong to a fixed set $A$ with packing dimension at most  $s \in [0, n]$. In this case, we first show that the Hausdorff dimension of  the Kakeya set is at least $n - s$. Furthermore, using Bourgain's bush argument, we improve the lower bound to $\max \paree{n - s, n - g_n(s)}$, where $g_n(s)$ is defined inductively. For example, when $n = 4$, we prove that the Hausdorff dimension  is at least $\max\{\frac{19}{5} - \frac{3}{5}s,4-s\}$. We also establish Kakeya maximal function analogues of these results.
\end{abstract}

\tableofcontents

\section{Introduction}

A \textit{Kakeya set} (or \textit{Besicovitch set}) is a compact subset of $\RR^n$ ($n \geq 2$) that contains a unit line segment in every direction. The formal definition is as follows: 

\begin{defn}
	A compact set $K \subseteq \RR^n$ is said to be a Kakeya set if, for all $e \in S^{n-1}$, there exists a point $a_e \in \RR^n$ such that 
	\[
	I_e(a_e) := \{a_e + t \cdot e : -1/2 \leq t \leq 1/2 \} \subseteq K.
	\]
	Here, $I_e(a_e)$ denotes the unit line segment in the direction $e$ with midpoint $a_e$. 
\end{defn}

In the early 20th century, Besicovitch showed the existence of a Kakeya set in $\RR^n$ with zero Lebesgue measure. Following this, researchers began studying finer properties of Kakeya sets, such as their Hausdorff and box-counting dimensions. For convenience, we provide the definitions of these dimensions below.

\begin{defn}
	Let $E$ be a bounded subset of $\RR^n$.
	\begin{enumerate}[(a)]
		\item The \textit{Hausdorff dimension} of $E$ is defined as
		\[
		\haus E = \inf\left\{\alpha : \forall \varepsilon > 0, \, \exists \{E_i\}_{i=1}^\infty \text{ such that } E \subseteq \bigcup_{i=1}^\infty E_i \text{ and } \sum_{i=1}^\infty \mathrm{diam}(E_i)^\alpha < \varepsilon \right\}.
		\]
		\item For any $\delta > 0$, let $N_\delta(E)$ denote the smallest number of sets with diameter at most $\delta$ needed to cover $E$. The \textit{lower box-counting dimension} of $E$ is defined as 
		\[
		\lbox E = \liminf_{\delta \to 0} \frac{\log N_\delta(E)}{-\log \delta}.
		\]
		Similarly, the \textit{upper box-counting dimension} of $E$ is defined as 
		\[
		\ubox E = \limsup_{\delta \to 0} \frac{\log N_\delta(E)}{-\log \delta}.
		\]
		\item Based on the box-counting dimension, the \textit{packing dimension} of $E$ is defined as
		\[
		\pacd E = \inf\left\{\sup_i \ubox E_i : E \subseteq \bigcup_{i=1}^\infty E_i\right\}.
		\]
	\end{enumerate}
\end{defn}

It is straightforward to verify the following relationships between these dimensions. For more details, see \cite{falconer} and \cite{M15}:
\begin{align*}
	&0 \leq \haus E \leq \pacd E \leq \ubox E, \\
	&0 \leq \haus E \leq \lbox E \leq \ubox E. \nonumber
\end{align*}
The \emph{Kakeya conjecture} asserts that (despite the existence of Kakeya sets of zero Lebesgue measure) Kakeya sets must be large in terms of dimension.

\begin{conj}[Kakeya conjecture] \label{kconjecture}
	For any Kakeya set $K$ in $\RR^n$, its Hausdorff dimension satisfies 
	\begin{equation}
		\haus K = n.\nonumber
	\end{equation}
	\end{conj}
Davies    solved the planar case in 1971 \cite{D71}.  In 1991, Bourgain \cite{B91} used the ``bush argument" to show that a Kakeya set in $\RR^3$ has Hausdorff dimension at least $\frac{7}{3}$. He also extended his results to higher dimensions via induction. Later, Wolff \cite{W95} improved Bourgain's result in 1995, establishing a lower bound of $\frac{n+2}{2}$ for the Hausdorff dimension in the general $\RR^n$ case, using a method some refer to as the ``hairbrush argument". In particular, the ``hairbrush argument" demonstrated that the Hausdorff dimension of Kakeya sets in $\RR^3$ is at least $\frac{5}{2}$. In 2002, Katz and Tao \cite{KT02} further improved Wolff's bound to $(2-\sqrt{2})(n-4) + 3$ for  $n \geq 5$. 
Recently, Wang and Zahl \cite{WZ25} resolved the conjecture in $\mathbb{R}^3$; a major breakthrough in the field. The problem, however, remains open in $\mathbb{R}^n$ for $n \geq 4$.

One may think of a Kakeya set in the following way. Given any direction $e \in S^{n-1}$, there exists a translate $a_e \in \mathbb{R}^n$ such that  the unit line segment $\{t \cdot e: -1/2 \leq t \leq 1/2\}$ is translated into the set by $a_e$ (that is, for each direction, one is free to choose the midpoint of the line associated with that direction).  One of the difficulties in proving the Kakeya conjecture is that one has no control over the choice of midpoints.  In this paper we impose some control on the choice of midpoints (thus making the problem easier) by insisting they all belong to a given set $A$.  We ask:~what conditions on $A$ are needed to obtain better dimension bounds than the current state-of-the-art?  Intuitively, the smaller $A$ is, the more heavily the Kakeya set is restricted and we quantify the size of $A$ via dimension. One might expect  a lower bound for the Hausdorff dimension which tends to $n$ as the packing dimension of $A$ tends to zero and we do indeed achieve this.  In fact this can be done quite easily and we first present two simple direct arguments, the second of which was provided by Tam\'as Keleti.  The main aim of the paper is to beat these initial estimates, and we are able to do this via a modification of Bourgain's bush argument.
\begin{propA} \label{keleti}
	Let $K$ be a Kakeya set in $\RR^n$ and let $A$ be the collection of all the  midpoints of the unit line segments contained in $K$. Then
	\begin{enumerate}
		\item If $\ubox A \leq s$, then $\lbox K \geq n-s$.
		\item If $\pacd A \leq s$, then $\haus K \geq n-s$. 
	\end{enumerate}
\end{propA}
\begin{proof}
	\begin{enumerate}
		  \item Fix $ \varepsilon>0$ and let $\delta\in (0,1)$. Since $\ubox A \leq s$,   $N_{\delta}(A) \leq c_1 \delta^{-(s + \varepsilon)}$ for some constant $c_1>0$ independent of $\delta$. Let $S_{\delta}$ be a $100\delta$-separated subset of $S^{n-1}$ of size  at least $c_2 \delta^{1-n}$  for another constant $c_2>0$ independent of $\delta$. For each $e \in S_{\delta}$, there is a unit line segment $I_e(a_e)$ in $K$ with  midpoint $a_e \in A$. By the pigeonhole principle, there exists a ball $B$ of diameter $\delta$ that contains at least $\frac{c_2\delta^{1-n}}{N_{\delta}(A)} \geq \tfrac{c_2}{c_1} \delta^{1-n+s+\varepsilon}$ points in $\{ a_e : e \in S_{\delta} \}$. Therefore, using the separation of directions in $S_\delta$, the union of all the unit line segments  whose midpoints are contained in the ball $B$  needs at least $\tfrac{c_2}{2c_1}\frac{\delta^{1-n+s+\varepsilon}}{\delta} =\tfrac{c_2}{2c_1} \delta^{-n+s+\varepsilon}$ many sets of diameter $\delta$  to cover it.  The claimed lower bound  follows upon letting $\varepsilon \to 0$.
		\item   The set $ K - A $ contains a solid ball of radius $ 1/2$. Moreover, $K-A$ is a Lipschitz image (an orthogonal projection) of the Cartesian  product $K \times A$. Therefore, 
\[
n = \haus(K - A) \leq \haus(K \times A) \leq \haus K + \pacd A,
\]
where the last inequality is a general result about products, which can be found in  \cite{falconer, M95}.  The claimed lower bound  follows by rearranging. 
	\end{enumerate}
\end{proof}

Our main dimension result can be found in Corollary \ref{packing-arbitrary}; see also Corollary~\ref{n-s}, Corollary~\ref{n-gs}, Corollary~\ref{generalcor_ub}, and Corollary~\ref{packing}.  A special case of this considers the midpoints of the unit line segments and provides an estimate for the Hausdorff dimension of a restricted Kakeya set, which beats the lower bound $n - \pd A$ from Proposition A. The same result holds for any arbitrary points in each unit line segment.  Importantly it also beats the state-of-the-art bound for the general Kakeya conjecture in dimension at least 4  and the bound $ n-\pd A$ \emph{simultaneously}   for some range of $\pd A$.  However, regrettably we do not know how to beat the bound $ n-\pd A$  for $\pd A$ close to zero. Explicit examples when $n=4$ and $n=10$ can be found in Figure \ref{n=4} and Figure \ref{n=10}.  

 To study the Hausdorff dimension of a Kakeya set, one often considers a maximal function and estimates its norm. We also provide appropriate maximal function versions  of our dimension results; see Theorem \ref{n-sthm} and Theorem \ref{mainresult}.   Our dimension results are obtained as applications of these maximal function estimates. Here we briefly recall the classical setting.
 
\begin{defn}
	Given any $f \in L^1_{loc}(\RR^n)$, the \textit{Kakeya maximal function of $f$} is a function $(f)_{\delta}^*: S^{n-1} \rightarrow \RR$ defined by 
	\mm{(f)_{\delta}^*(e)=\underset{a \in \RR^n}{\sup} \frac{1}{\abso{\tube{e}{a}}}\intt{\tube{e}{a}}{}\abso{f(x)}dx,}
	where the $\delta$-tube $\tube{e}{a_e}$ is the $\delta$-neighbourhood of the unit line segment $I_{e}(a_e)$.
	
	When the two vectors $e_1, e_2$ satisfy $\abso{e_1-e_2} > \delta$, we say $e_1$ and $e_2$ are $\delta$-separated, and the tubes $\tube{e_1}{a_1}, \tube{e_2}{a_2}$ are called $\delta$-separated tubes.
\end{defn}

The following conjecture, known as the Kakeya maximal function conjecture, implies the Kakeya conjecture.
\begin{conj}[Kakeya maximal function conjecture] \label{kmconjecture}
	For all $\varepsilon >0$, there exists a constant $C_{n,\varepsilon}$ only depending on $n$ and $\varepsilon$ such that for all $f \in L^1_{loc}(\RR^n)$ and $0 < \delta < 1$,
	\begin{equation}
		\nnn{(f)_{\delta}^*}{L^n(S^{n-1})}{} \leq C_{n,\varepsilon} \delta^{-\epsilon}\nnn{f}{L^n(\RR^n)}{}. \nonumber
	\end{equation} 
\end{conj}
One can see the proof of Conjecture \ref{kmconjecture} implying Conjecture \ref{kconjecture} in \cite[Proposition 10.2]{wlecturenotes}. The planar case of Conjecture \ref{kmconjecture}  was proved by C\'ordoba \cite{C77}, but it remains open in $\RR^n$, $n \geq 3$.

  The main theorem of this paper, Theorem \ref{mainresult}, demonstrates the connection between estimates for the Kakeya maximal function in $\RR^{n-1}$ and our   \emph{restricted}  Kakeya maximal function  in $\RR^n$; see Definition \ref{defrestricted} for the precise definition. Given the breakthrough on the Kakeya conjecture in $\RR^3$ by Hong and Zahl \cite{WZ25}, we focus more on the case of $\RR^4$, rather than $\RR^3$, when we consider examples  in Section \ref{examples}. For general $n \geq 5$, we apply the results of Hickman, Rogers, and Zhang \cite{HRZ22} on the Kakeya maximal conjecture to obtain lower bounds for the Hausdorff dimension of restricted Kakeya set in Corollary \ref{generalcor_ub}.

\section{Main Results}

\subsection{Key definitions and main theorems}
In this paper, we use $|\cdot|$ to denote general volume measure and we often use a subscript to indicate the dimension if this is not clear from context.  More precisely, $|\cdot|_{n}$ will represent the Lebesgue measure in $\RR^n$, $|\cdot|_{n-1}$ will  represent the surface measure on the sphere $S^{n-1}$ and $|\cdot|_{1}$ also means the length in $\RR^n$. We use the notation $ A \lesssim B $ to indicate that $ A \leq C_n B $, where the constant $ C_n $ depends only on the ambient spatial dimension $ n $. We write $A \approx B $ if there exist two positive constants $c_n$ and $C_n$ depending on $n$ such that $c_n A \leq B \leq C_n A$. Similarly, we write $ A \lesssim_{\varepsilon} B $ to mean that $ A \leq C_{n,\varepsilon} B $, where the constant $ C_{n,\varepsilon} $ depends on both $ n $ and a parameter  $ \varepsilon $. For any Lebesgue measurable subset $E \subseteq \mathbb{R}^n$, $\chi_E$ denotes the characteristic function of $E$.

We first  define a   special type of Kakeya set where the  midpoints of each line segment are restricted to a given set $A$.  These sets are our main object of study.   
\begin{defn} Given a bounded set $A \subseteq \RR^n$, we say a compact set $K_A$ is  an \emph{$A$-restricted Kakeya set}  if, for all  $e \in S^{n-1}$, there exists a point $a_e \in A$ such that the unit line segment $I_e(a_e)$, in the direction $e$ with midpoint $a_e$, is contained in $K_A$. 
\end{defn}
We also consider the analogous $A$-restricted Kakeya maximal functions.  Using the ``bush argument" given by Bourgain \cite{B91}, we prove a weak-type inequality related to these  maximal functions. 
\begin{defn}\label{defrestricted}
	Given any $A \subseteq \RR^n$ and $f \in L^1_{loc}(\RR^n)$, the \textit{$A$-restricted Kakeya maximal function of $f$} is a function $\McK{f}: S^{n-1} \rightarrow \RR$ defined by 
	\mm{\McK{f}(e)=\underset{a \in A}{\sup} \frac{1}{\abso{\tube{e}{a}}}\intt{\tube{e}{a}}{}\abso{f(x)}dx.}
\end{defn}

We define this $A$-restricted Kakeya maximal function because the restricted weak-type $(p, q)$ estimates for it will yield a lower bound on the Hausdorff dimension of $K_A$. The following lemma, analogous to \cite[Proposition 10.2]{wlecturenotes}, establishes the connection between the Kakeya maximal function and the Hausdorff dimension of $K_A$, and one can check the proof in Section \ref{lemmaproof}.
\begin{lma} \label{keylemma}
  Suppose that for some $1 \leq p \leq q < \infty$ and $\beta > 0$, 
 \begin{align}
  \nnn{\McK{\chi_E}}{L^{q,\infty}(S^{n-1})}{} \lesssim_{\varepsilon} \delta^{-\beta-\varepsilon} \nnn{\chi_E}{L^{p}(\mathbb{R}^n)}{},  \label{ineq}
\end{align} 
for all measurable sets $E \subseteq \RR^n$, $0<\delta <1$ and $\varepsilon >0$. Then the Hausdorff dimension of every $A$-restricted Kakeya set   is at least $n - \beta p$. 
\end{lma}

We now state our first main theorem.  It establishes a  weak type estimate for the $A$-restricted Kakeya maximal function.   The proof is  presented in Section \ref{babybushproof}.

\begin{thm} \label{n-sthm}
For $A \subseteq \RR^n$ with $\ubox A \leq s$,  
\begin{align} 
\nnn{\McK{\chi_E}}{L^{n,\infty}(S^{n-1})}{} \lesssim_{\varepsilon} \delta^{-\frac{s}{n}-\varepsilon} \nnn{\chi_E}{L^n(\mathbb{R}^n)}{} \label{n-sformula}
\end{align} 
holds for all Lebesgue measurable sets $E \subseteq \RR^n$ and all positive $\delta $. 
\end{thm}

As an immediate corollary of the previous theorem, we obtain an alternative proof of the simple estimate from Proposition A (2). Even though this is stated in terms of the box dimension of $A$, it is straightforward to upgrade this to packing dimension; see Corollary \ref{packing}.
\begin{cor} \label{n-s}
Let $A \subseteq \RR^n$ with $\ubox A \leq s$ and $K_A$ be  an $A$-restricted Kakeya set.  Then $$\haus K_A \geq n - s.$$
\end{cor}

\begin{proof} 
The result follows by  Lemma \ref{keylemma} taking $p = q = n$ and $\beta = \frac{s}{n}$ and using the estimate from Theorem \ref{n-sthm}. 
\end{proof}

Our  next result is proved by  an application of Bourgain's bush argument \cite{B91}. Given an estimate  for the Kakeya maximal function in dimension $n-1$,   we derive an estimate for the $A$-restricted Kakeya maximal function in $\RR^n$. Note that the assumption \eqref{result_n-1} is necessary in the proof of Lemma \ref{bourgainlemma}; see the proof of \cite[Lemma 1.52]{B91} for further details.  This is where we need information  about the Kakeya maximal function in dimension $n-1$.  Essentially it comes from slicing $\mathbb{R}^n$ by hyperplanes and using unrestricted  estimates on the hyperplanes. 

\begin{thm} \label{mainresult}
	Suppose for some $h_{n-1} >0$ and $p_{n-1} >1$, 
	\begin{equation}
		\nnn{(f)^*_{\delta}}{L^{p_{n-1}}(S^{n-2})}{} \lesssim_{\varepsilon} C \delta^{-h_{n-1}-\varepsilon}\nnn{f}{L^{p_{n-1}}(\mathbb{R}^{n-1})}{}\label{result_n-1} 
	\end{equation}
	holds for all $f \in L^1_{loc}(\RR^{n-1})$ and all $\varepsilon >0$. Then for $A \subseteq \RR^n$ with $\ubox A \leq s$ and all $\varepsilon>0$,
	\begin{equation}
		\nnn{\McK{\chi_E}}{L^{p,\infty}(S^{n-1})}{} \lesssim_{\varepsilon} \delta^{-\beta-\varepsilon} \nnn{\chi_E}{L^p(\mathbb{R}^n)}{}, \label{mainresultformula}
	\end{equation}
	where \begin{equation}
		p=\frac{p_{n-1}+n(p_{n-1}-1)+1}{p_{n-1}}, \quad \beta=\frac{h_{n-1}p_{n-1}+sp_{n-1}-s}{p_{n-1}+n(p_{n-1}-1)+1}.
	\end{equation}
\end{thm}
Theorem \ref{mainresult} is proved in Section \ref{bushproof}. By Lemma \ref{keylemma} and Theorem \ref{mainresult}, we obtain the following corollary concerning  the Hausdorff dimension of $K_A$.
\begin{cor} \label{n-gs}
	Suppose \eqref{result_n-1} holds in $\RR^{n-1}$ for some  $h_{n-1} >0$ and $p_{n-1} >1$. Let $$g_n(s)=\frac{h_{n-1}p_{n-1}+sp_{n-1}-s}{p_{n-1}}.$$ If $A \subseteq \RR^n$ with $\ubox A \leq s$ and $K_A$ is an $A$-restricted Kakeya set in $\RR^n$, then 
	\begin{equation}
		\haus K_A \geq n-g_n(s). \nonumber
	\end{equation}
\end{cor}

In the above results we used upper box dimension to quantify the size of $A$.  However, in Corollary \ref{packing}, we extend our results by showing that the conclusions of Corollary \ref{n-s}, Corollary \ref{n-gs}, and (the later) Corollary \ref{generalcor_ub} remain valid when considering the packing dimension of $A$ instead of the upper box dimension. We  also prove that it is not necessary to restrict the \emph{mid}points of the  line segments. Instead, Corollary \ref{packing-arbitrary} allows $A$ to consist of an arbitrary point from each line segment, thus giving the approach more flexibility. 

\subsection{A-Restricted Kakeya sets  in $\RR^3$ and $\RR^4$} \label{examples}

If we rewrite \eqref{result_n-1} in $\RR^n$ as
\begin{equation}
	\nnn{(f)^*_{\delta}}{L^{q_0}(S^{n-1})}{} \lesssim_{\varepsilon} \delta^{-\frac{n}{p_0}+1-\varepsilon}\nnn{f}{L^{p_0}(\mathbb{R}^{n})}{} \label{classicalresult_1}
\end{equation}
for some $p_0 \geq 1$ and $q_0 \geq p_0$, then we can interpolate with the trivial estimate
\begin{equation}
	\nnn{(f)^*_{\delta}}{L^{\infty}(S^{n-1})}{} \lesssim \delta^{-(n-1)}\nnn{f}{L^{1}(\mathbb{R}^{n})}{} \nonumber
\end{equation}
to obtain
\begin{equation}
	\nnn{(f)^*_{\delta}}{L^{q}(S^{n-1})}{} \lesssim_{\varepsilon} \delta^{-\frac{n}{p}+1-\varepsilon}\nnn{f}{L^{p}(\mathbb{R}^{n})}{} \label{classicalresult_2}
\end{equation}
for all $1 \leq p \leq p_0$ and $$q=q_0\frac{1-\frac{1}{p_0}}{1-\frac{1}{p}} \geq q_0 \geq p_0 \geq p.$$ If Conjecture \ref{kmconjecture} holds in $\RR^n$, then \eqref{classicalresult_1} holds for $p_0=q_0=n$.\\

First, let us consider the case $n=3$.  Conjecture \ref{kmconjecture} is known to be true in $\RR^2$. Thus, \eqref{classicalresult_2} holds when $n=2$, i.e.,
\begin{equation}
	\nnn{(f)^*_{\delta}}{L^{q}(S^{1})}{} \lesssim \delta^{-\frac{2}{p}+1-\varepsilon}\nnn{f}{L^{p}(\mathbb{R}^{2})}{} \nonumber
\end{equation}
for all $1 \leq p \leq 2$. Using this result in $\mathbb{R}^2$,  Corollary \ref{n-gs}  gives the following lower bound for the Hausdorff dimension of $K_A$ in the case $n=3$:  
\begin{equation}
	\haus K_A \geq \max_{1 \leq p \leq 2} \paree{3-\frac{2-s}{p}-(s-1)}.\nonumber
\end{equation}
In particular, taking $p=2$ yields
\[
\haus K_A \geq 3-\frac{s}{2}.
\]
Of course this result is obsolete, given that   Wang and Zahl \cite{WZ25} proved that $\haus K  = 3$ for \emph{all} Kakeya sets.\\

Next, we consider the case $n=4$, where the Kakeya conjecture is still open. We first apply Wolff's result \cite{W95}  in $\RR^3$, which gives the estimate: 
\begin{equation}
	\nnn{(f)^*_{\delta}}{L^{q}(S^{2})}{} \lesssim \delta^{-\frac{3}{p}+1-\varepsilon}\nnn{f}{L^{p}(\mathbb{R}^{3})}{} \nonumber
\end{equation}
for all $1 \leq p \leq \frac{5}{2}$. Therefore, by applying Corollaries \ref{n-s} and \ref{n-gs}, we obtain the following lower bound for the Hausdorff dimension of  $A$-restricted Kakeya sets in $\RR^4$:
\begin{equation}
\begin{split}
	\haus K_A &\geq \max_{1 \leq p \leq \frac{5}{2}} \paree{4-\frac{3-s}{p}-(s-1),4-s} \\ &=
	\begin{cases}
4-s, &\quad 0 \leq s< \frac{1}{2}\\
\frac{19}{5} - \frac{3}{5}s , &\quad \frac{1}{2} \leq s< 3\\
2, &\quad 3 \leq s \leq4. \\
\end{cases} \nonumber
\end{split}
\end{equation}
The relation between the lower bound of $\haus K_A$ and $s$ is illustrated in Figure \ref{n=4}.
\begin{figure}[htbp]
		\centering
		\includegraphics[scale=0.62]{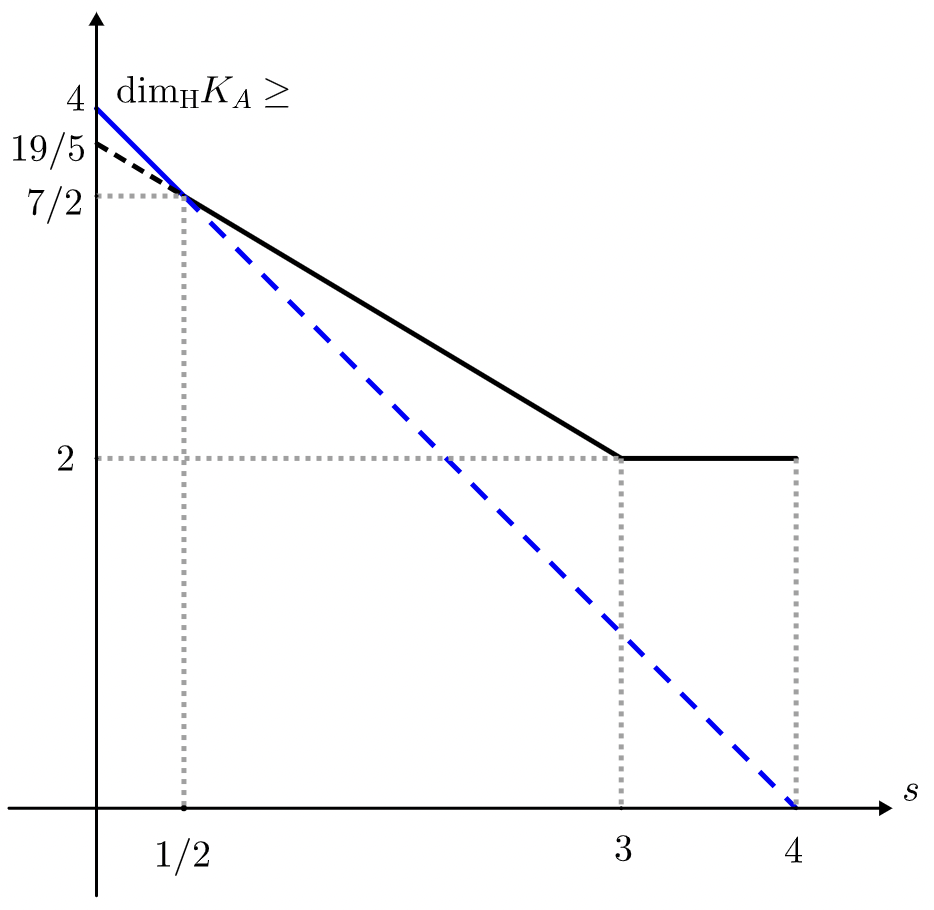} 
		\caption{Lower bound for $\haus K_A$ in $\RR^4$ under the assumption that $\ubox A \leq s$. The blue bound comes from Corollary \ref{n-s} and the black bound comes from Corollary \ref{n-gs}. The best lower bound is the maximum of these two and is shown as a solid line.}
		\label{n=4}
\end{figure}
Note that every Kakeya set in $\mathbb{R}^4$ must have  Hausdorff dimension   at least $3.059$, as shown by Katz and Zahl \cite{KZ19+}, and this is the current state-of-the-art. In particular,  $3.059<7/2$, and so both Corollaries \ref{n-s} and \ref{n-gs} are needed to obtain the best possible information; see Figure \ref{n=4}.

\subsection{Higher Dimensions}
According to \cite[Proposition 22.6]{M15}, we have the following discrete version of \eqref{result_n-1}.
\begin{prop} \label{discreteprop}
	Let $1 < p_{n-1} < \infty$, $p_{n-1}'=\frac{p_{n-1}}{p_{n-1}-1}$, $h_{n-1}>0$ and $0 < \delta < 1$. Then
	\begin{equation}
		\nnn{(f)^*_{\delta}}{L^{p_{n-1}}(S^{n-2})}{} \lesssim_{\varepsilon}  \delta^{-h_{n-1}-\varepsilon}\nnn{f}{L^{p_{n-1}}(\mathbb{R}^{n-1})}{} \nonumber
	\end{equation}
	for all $f \in L^1_{loc}(\RR^{n-1})$, $\varepsilon >0$, if and only if 
	\begin{equation}
		\bignnn{\summ{k=1}{m}\chi_{T_k}}{L^{p_{n-1}'}(\RR^{n-1})}{} \lesssim_{\varepsilon}\delta^{-h_{n-1}-\varepsilon}\pare{\summ{k=1}{m}\abso{T_k}_{n-1}}^{\frac{1}{p_{n-1}'}} \label{discreteversion}
	\end{equation}
	for all $\delta$-separated $\delta$-tubes $T_1, \dots , T_m$ and for all $\varepsilon>0$.
\end{prop}
Next, we apply the result about Kakeya maximal function in $\RR^{n-1}$ from Hickman, Rogers and Zhang \cite{HRZ22}.
\begin{thm}[\cite{HRZ22}]\label{hickman}
Let 
\begin{equation}
	w(n-1)=1+\min_{\substack{2 \leq t \leq n-1\\ t \in \mathbb{N}}} \max \paree{\frac{2(n-1)}{(n-2)(n-1)+(t-1)t},\frac{1}{n-t}}. \label{w(n-1)}
\end{equation}
Then for all $\varepsilon >0$ and $0 < \delta <1$, 
	 \begin{equation}
	 	\bignnn{\summ{k=1}{m}\chi_{T_k}}{L^{p}(\RR^{n-1})}{} \lesssim_{\varepsilon}\delta^{-(n-2-\frac{n-1}{p})-\varepsilon}\pare{\summ{k=1}{m}\abso{T_k}_{n-1}}^{\frac{1}{p}}\nonumber
	 \end{equation}
	 when $p \geq w(n-1)$. 
\end{thm} 
Let $w(n-1)' $ be the dual index of $w(n-1)$ satisfying $\frac{1}{w(n-1)}+\frac{1}{w(n-1)'}=1$. Therefore, \eqref{discreteversion} is true with 
\begin{equation}
	h_{n-1}=n-2-\frac{n-1}{p_{n-1}'}\nonumber
\end{equation} 
whenever $p_{n-1}' \geq w(n-1)$, where $p_{n-1}'$ is the dual index of $p_{n-1}$. Note that $p_{n-1}' \geq w(n-1)$ is equivalent to $p_{n-1} \leq w(n-1)'$. By Proposition \ref{discreteprop}, we obtain the result for Kakeya maximal function in $\RR^{n-1}$.
\begin{thm} \label{general_n-1}
	Let $h_{n-1}=(n-2-\frac{n-1}{p_{n-1}'})$. Then 
	\begin{equation}
		\nnn{(f)^*_{\delta}}{L^{p_{n-1}}(S^{n-2})}{} \lesssim_{\varepsilon} C \delta^{-h_{n-1}-\varepsilon}\nnn{f}{L^{p_{n-1}}(\mathbb{R}^{n-1})}{}\nonumber
	\end{equation}
	holds for all $p_{n-1} \leq w(n-1)'$.
\end{thm}
Using this and Corollary \ref{n-gs}, together with Corollary \ref{n-s}, we have the explicit lower bound of $\haus K_A$ in $\RR^n$.

\begin{cor} \label{generalcor_ub}
	Let $A \subseteq \RR^n$ with $\ubox A \leq s$ and $K_A$ be an $A$-restricted Kakeya set in $\RR^n$.  Then \begin{equation}
	\begin{split}
\haus K_A &\geq \max_{1 \leq p \leq w(n-1)' } \paree{n-\frac{n-1-s}{p}-(s-1),n-s}\\ &=\begin{cases}
n-s, &\, 0 \leq s< n-1-w(n-1)'\\
n-\frac{n-1-s}{w(n-1)'}-(s-1) , & n-1-w(n-1)' \leq s< n-1\\
2, & n-1 \leq s \leq n,
\end{cases}
\end{split}\nonumber
\end{equation}
where $w(n-1)'$ is the dual index of $w(n-1)$ in \eqref{w(n-1)}.
\end{cor}
For example when $n=10$, from \cite[Figure 1]{HRZ22}, $w(9)'=6$. Now our lower bound for $\haus K_A$ in $\RR^{10}$ is
\begin{equation}
	\begin{split}
\haus K_A \geq \begin{cases}
10-s, &\, 0 \leq s< 3\\
\frac{19}{2}-\frac{5}{6}s , & 3 \leq s< 9\\
2, & 9 \leq s \leq 10,
\end{cases}
\end{split}\nonumber
\end{equation}
see Figure \ref{n=10}.
\begin{figure}[htbp]
		\centering
		\includegraphics[scale=0.68]{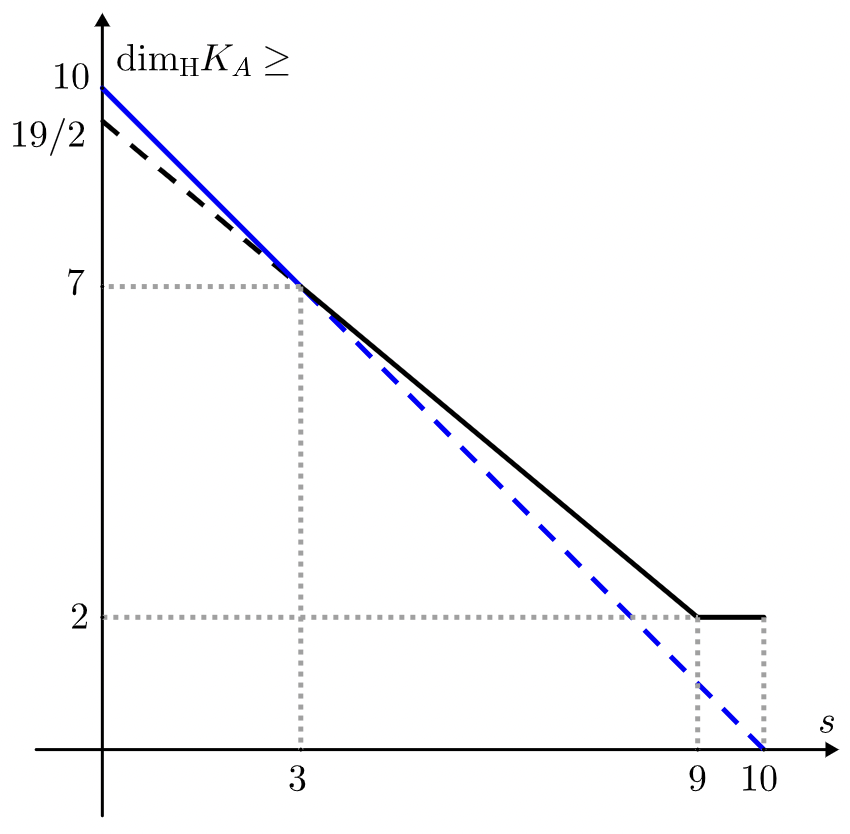} 
		\caption{Lower bound for $\haus K_A$ in $\RR^{10}$ under the assumption that $\ubox A \leq s$. The blue bound comes from Corollary \ref{n-s} and the black bound comes from Corollary \ref{n-gs}. The best lower bound is the maximum of these two and is shown as a solid line.}
		\label{n=10}
\end{figure}
In general,   the Hausdorff dimension of a Kakeya set in $\mathbb{R}^{10}$ is at least  $15 - 6\sqrt{2} \approx 6.515$, as given by Katz and Tao \cite{KT02}.   This is again relevant because it means both of our lower bounds are needed; see  Figure \ref{n=10}.

\subsection{Further Remarks}
First, we can relax the upper box-counting dimension assumption on $A$ in Corollary \ref{n-s}, Corollary \ref{n-gs} and Corollary \ref{generalcor_ub} to the packing dimension of $A$. For convenience, we refer to  the lower bound for $\haus K_A$ in these three corollaries by a single  function $f(n,s)$ for $n \geq 3$. It is easy to see that in each case the function $f(n,s)$ is continuous with respect to $s$ when $n$ is fixed.

\begin{cor}\label{packing}
	Suppose $A \subseteq \mathbb{R}^n$. The same results for $\haus K_A$ in Corollary \ref{n-s}, Corollary \ref{n-gs} and Corollary \ref{generalcor_ub} hold if we replace $\ubd A \leq s$ by $\pacd A \leq s$.
\end{cor}

\begin{proof}
	Let  $K_A$ be  an $A$-restricted Kakeya set. Since $\pacd A \leq s$, for any $\varepsilon > 0$, there exists a covering $\{A_i\}_{i=1}^{\infty}$ of $A$ such that for each $i$,
	\begin{equation}
		\ubox A_i \leq s + \varepsilon.\nonumber
	\end{equation}
Fix an arbitrary $A_i$, and let $E_i \subseteq S^{n-1}$ be the set of directions for which the corresponding line segments have midpoints in $A_i$:
\begin{equation}
E_i = \{ e : I_e(a_e) \subseteq K_A \text{ and } a_e \in A_i \}.\nonumber
\end{equation}	
Since
\begin{equation}
	\bigcup_{i=1}^{\infty} E_i = S^{n-1},\nonumber
\end{equation}
there exists some $E_k$ such that 
\begin{equation}
	|E_k|_{n-1} > 0.\nonumber
\end{equation}	
By Lebesgue's density theorem on $S^{n-1}$, there exist finitely many elements $r_1, \dots, r_N$ in the orthogonal group of $S^{n-1}$ such that 
\begin{equation}
	\left| \bigcup_{i=1}^{N} r_i(E_k) \right|_{n-1} > \frac{1}{2}.\nonumber
\end{equation}
In fact, there exists a point $x_0 \in E_k$, such that for sufficiently small radius $\delta$, the open ball $B(x_0,\delta)$ satisfies
\begin{equation}
	\frac{\abso{B(x_0,\delta)\cap E_k}_{n-1}}{\abso{B(x_0,\delta)}_{n-1}} > \frac{99}{100}. \nonumber
\end{equation}
For some large $N\lesssim \delta^{1-n}$, we then choose $r_1, \dots, r_N \in O(n)$ such that $\paree{r_i(B(x_0,\delta)) \cap S^{n-1}}_{i=1}^{N}$ are disjoint. Consequently, we obtain the estimate  
\begin{align*}
	\abso{\bigcup_{i=1}^{N} r_i(E_k)}_{n-1} & \geq \abso{\bigcup_{i=1}^{N} r_i(B(x_0,\delta)\cap E_k)}_{n-1}\\
	&=\summ{i=1}{N} \abso{r_i(B(x_0,\delta)\cap E_k)}_{n-1}\\
	& > \frac{99}{100}N \abso{B(x_0,\delta)\cap S^{n-1}}_{n-1}\\
	& > \frac{1}{2}
\end{align*}
provided  $N$ is large enough so that $N \abso{B(x_0,\delta)\cap S^{n-1}}_{n-1} > \frac{2}{3}$.
Let $K_A^k$ be the subset of $K_A$ consisting of line segments with directions in $E_k$. Let $E_0 = \bigcup_{i=1}^{N} r_i(E_k)$, $A_0 = \bigcup_{i=1}^{N} r_n(A_k)$ and $K_0 = \bigcup_{i=1}^{N} r_n(K_A^k)$. One can see the collection of midpoints in $K_0$ is $A_0$ and the collection of directions in $K_0$ is $E_0$ with $\abso{E_0}_{n-1} > \frac{1}{2}$. Then
\begin{equation}
	\ubox A_0 = \ubox A_k \leq s + \varepsilon.\nonumber
\end{equation}
By Corollary \ref{n-s}, Corollary \ref{n-gs} or Corollary \ref{generalcor_ub}, we obtain $\haus K_0 \geq f(n,s+\varepsilon)$. Since $K_0$ is a union of finitely many copies of $K_A^k$, it follows that
\begin{equation}
	\haus K_A \geq \haus K_A^k = \haus K_0 \geq f(n,s+\varepsilon).\nonumber
\end{equation}
Letting $\varepsilon \to 0$ completes the proof.
\end{proof}

The next corollary shows that it is unnecessary to restrict the \emph{mid}points to $A$. In fact, we can take one point from each unit line segment at \emph{any} position, and we will still obtain the same lower bound for the Hausdorff dimension.

\begin{cor} \label{packing-arbitrary}
	Suppose $K$ is a Kakeya set in $\mathbb{R}^n$, and let $P \subseteq K$ be a set such that for each $I_e(a_e) \subseteq K$, the intersection $P \cap I_e(a_e)$ is nonempty.  
	If $\pacd P \leq s$, then $\haus K \geq f(n,s)$, where the function $f(n,s)$ refers to the lower bound  in Corollary \ref{n-s}, Corollary \ref{n-gs} or Corollary \ref{generalcor_ub}.
\end{cor}

\noindent \emph{Proof sketch.} 
First observe that all of our arguments go through if we replace \emph{mid}point with \emph{end}point,   consider line segments of length 1/2 instead of length 1, and replace all directions with a set of directions of positive measure.  Indeed, we only chose to consider the midpoints for some aesthetic reasons. Second, for each unit line segment $ I_e(a_e) $ in the Kakeya set $ K $, let $ x_e \in P \cap I_e(a_e) $. There exists a sub-segment $ I'_e(a_e) $ of length $ \frac{1}{2} $ within $ I_e(a_e) $, with one of its endpoints at $ x_e $. Defining $ K' $ as the union of all such $ \frac{1}{2} $-length segments, we have $ K' \subseteq K $, and $ P $ becomes the collection of endpoints of each segment in $ K' $.  The desired  conclusion follows   since $\haus K \geq \haus K'$.




	Alternatively, we could introduce a new $A$-restricted Kakeya maximal function defined as  
	\begin{equation}
		\mathcal{K}_{\delta,P}^*(f)(e)=\sup_{\substack{-\frac{1}{2} \leq t \leq \frac{1}{2}\\ a_e \in P+t \cdot e}}\frac{1}{\abso{\tube{e}{a_e}}}\intt{\tube{e}{a}}{}\abso{f(x)}dx. \nonumber
	\end{equation} 
	This operator allows the tubes to shift along the direction $e$ based on the set $P$. It is straightforward to verify that all relevant lemmas and theorems still hold for this new maximal function. \hfill \qed \\

Finally, the next corollary gives a new sufficient condition  for the Kakeya conjecture to hold.
\begin{cor} \label{kakeyacor}
	For any Kakeya set $K \subseteq \mathbb{R}^n$, for all $\varepsilon > 0$, if there exists a subset $P$ of $K$ with $\pacd P < \varepsilon$ such that for each $I_e(a_e) \subseteq K$, the intersection $P \cap I_e(a_e)$ is nonempty, then $\haus K = n$.
\end{cor}

\begin{proof}
	This follows directly from Corollary \ref{n-s} and Corollary \ref{packing-arbitrary}.
\end{proof}

\section{Remaining Proofs} \label{section3}

\subsection{Proof of Lemma \ref{keylemma}} \label{lemmaproof}

The inequality \eqref{ineq} is equivalent to 
	\begin{align} \abso{\paree{e \in S^{n-1} : \McK{\chi_E}(e) > \lambda}}_{n-1} \lesssim_\varepsilon  \pare{\lambda^{-1} \delta^{-\beta-\varepsilon} |E|^{1/p}_n}^q, \label{equivalent2.1}
	\end{align}
	for all measurable sets $E \subseteq \RR^n$, $0<\lambda <1$, $\varepsilon>0$ and $0< \delta <1$.
	Suppose $K_A$ is  an $A$-restricted Kakeya set . Let $\{B_j\}_{j=1}^{\infty}$ be a covering of $K_A$, where $B_j = B(x_j, r_j)$ is an open ball. We can assume that $r_j \leq \frac{1}{100}$ for all $j \in \mathbb{N}$. Define $J_k = \paree{j : \frac{1}{2^k} \leq r_j < \frac{2}{2^k} }$. For every $e \in S^{n-1}$, $K_A$ contains a unit line segment $I_e(a_e)$ parallel to $e$ with $a_e \in A$. Let $S_k = \paree{e \in S^{n-1} : |I_e(a_e) \cap \cup_{j \in J_k} B_j |_1 \geq \frac{1}{100k^2}}$. Then we can see $\cup_{k=1}^\infty S_k = S^{n-1}$. In fact, if there exists an $e' \in  S^{n-1}$, but $e' \notin S_k$ for all $k$, then 
	\begin{equation}
		|I_{e'}(a_{e'}) \cap \cup_{j \in J_k} B_j |_1 < \frac{1}{100k^2} \nonumber
	\end{equation}
	for all $k$. Since $\{B_j\}_{j=1}^{\infty}$ is also a covering of $I_{e'}(a_{e'})$, 
	\begin{align*}
		\summ{k}{}|I_{e'}(a_{e'}) \cap \cup_{j \in J_k} B_j |_1 \geq |I_{e'}(a_{e'}) \cap \cup_{j=1}^{\infty} B_j |_1 =1.
	\end{align*}
	However, also
	\begin{align*}
		\summ{k}{}|I_{e'}(a_{e'}) \cap \cup_{j \in J_k} B_j |_1 < \summ{k}{}\frac{1}{100k^2} <1,
	\end{align*}
	which gives a contradiction. Therefore, $\cup_{k=1}^\infty S_k = S^{n-1}$
	
	Let $F_k = \cup_{j \in J_k} 10B_j$, where $10B_j$ denotes an open ball centred at the same point as $B_j$ but with a radius enlarged by a factor of 10. Define the function $f$ as the characteristic function of $F_k$, i.e., $f = \chi_{F_k}$.

For $e \in S_k = \{ e \in S^{n-1} : |I_e(a_e) \cap \cup_{j \in J_k} B_j |_1 \geq \frac{1}{100k^2} \}$, the intersection $F_k \cap T^{2^{-k}}_e(a_e)$ occupies a larger portion of $T^{2^{-k}}_e(a_e)$ than $I_e(a_e)$ intersecting $\cup_{j \in J_k} B_j$ within the line segment $I_e(a_e)$, which implies
	\mm{
	\frac{\abso{T^{2^{-k}}_e(a_e) \cap F_k}_n}{|T^{2^{-k}}_e(a_e)|_n} \gtrsim |I_e(a_e) \cap \cup_{j \in J_k} B_j |_1\gtrsim \frac{1}{k^2}.
	}
	Hence, when $e \in S_k$, 
	\mm{
	\cK_{2^{-k},A}(f)(e) \geq \frac{\abso{T^{2^{-k}}_e(a_e) \cap F_k}_n}{|T^{2^{-k}}_e(a_e)|_n}  \gtrsim \frac{1}{k^2}.
	}
	From $\eqref{equivalent2.1}$, 
	\mm{
	\abso{\paree{e \in S^{n-1}: \cK_{2^{-k},A}(f)(e)\gtrsim \frac{1}{k^2}}}_{n-1} \lesssim_{\varepsilon}  \pare{k^2 2^{k (\beta+\varepsilon)}|F_k|^{1/p}_n}^q. 
	}
	Thus,
	\mm{
	|S_k|_{n-1} \lesssim_{\varepsilon} k^{2q}2^{kq(\beta+\varepsilon)}(\#J_k)^{\frac{q}{p}}2^{-kn\frac{q}{p}}.
	}
	When $k$ is sufficiently large, $k^{2q} \leq 2^{k \varepsilon}$, so 
	\mm{
	|S_k|_{n-1} \lesssim_{\varepsilon} 2^{-k(n\frac{q}{p}-\beta q - (q+1)\varepsilon)}(\#J_k)^{\frac{q}{p}}.
	}
	Therefore, 
	\mm{
	 \#J_k 2^{-k(n-\beta p - \frac{p(q+1)}{q} \varepsilon)}\gtrsim_{\varepsilon} |S_k|_{n-1}^{\frac{p}{q}}.
	}
	Together with the definition of $J_k$, 
	\mm{
	\summ{j}{}r_j^{n-\beta p - \frac{p(q+1)}{q} \varepsilon} \gtrsim \summ{k}{}(\# J_k)2^{-k(n-\beta p - \frac{p(q+1)}{q} \varepsilon)} \gtrsim_{\varepsilon} \summ{k}{}|S_k|_{n-1}^{\frac{p}{q}} \gtrsim \summ{k}{} |S_k|_{n-1} \gtrsim 1.
	}
	By the definition of Hausdorff dimension, this gives the desired lower bound.
\subsection{Proof of Theorem \ref{n-sthm}} \label{babybushproof}

Before presenting the proof, we recall some simple yet useful geometric observations regarding $\delta$-tubes, as outlined in the proof of \cite[Proposition 11.8]{wlecturenotes}.

\begin{prop} \label{propgeo}
For all $e_1, e_2 \in S^{n-1}$ and all $a_1, a_2 \in \RR^n$, the following estimates hold for the diameter and measure of the intersection of two $\delta$-tubes: \begin{equation}
    \begin{split} 
    \textup{diam}\pare{\tube{e_1}{a_1} \cap \tube{e_2}{a_2}} &\lesssim \frac{\delta}{\theta(e_1, e_2) + \delta} < \frac{\delta}{\theta(e_1, e_2)}, \\
    \abso{\tube{e_1}{a_1} \cap \tube{e_2}{a_2}}_n &\lesssim \frac{\delta^n}{\theta(e_1, e_2) + \delta} < \frac{\delta^n}{\theta(e_1, e_2)}, \nonumber
     \end{split} 
     \end{equation}
     where $\theta(e_1, e_2)$ is the acute angle between $e_1$ and $e_2$.
  \end{prop}
In order to prove \eqref{n-sformula}, for any measurable set $E \subseteq \RR^n$, let
	\mm{
	E_{\lambda}=\paree{e \in S^{n-1}: \cK_{\delta,A}(\chi_{E})(e)> \lambda}.
	}
Then it suffices to show
\begin{equation}
	\lambda \abso{E_{\lambda}}_{n-1}^{\frac{1}{n}} \lesssim_{\varepsilon} \delta^{-\frac{s}{n}-\varepsilon} \abso{E}_n \label{equal2.2}
\end{equation}
for all $0< \delta <1$, $0< \lambda <1$ and $\varepsilon >0$.
	Choosing a maximal $\delta$-separated subset $\paree{e_1, \dots, e_N} \subseteq E_{\lambda}$, it is easy to see 
	\begin{equation}
		N \gtrsim \frac{|E_{\lambda}|_{n-1}}{\delta^{n-1}}. \label{sec3.2_1}
	\end{equation}
	For each $e_j$, there exists a midpoint $a_i \in A$ such that for the tube $T_j=T^{\delta}_{e_j}(a_j)$,
	\begin{equation}
		|E \cap T_j| > \lambda |T_j| \approx \lambda \delta^{n-1}. \label{thm1.2_1}
	\end{equation}
	By the assumption $\ubox A\leq s$, for any $\varepsilon >0$, there exists a covering of $A$ by $N_{\delta}$ open balls $\{B_i\}_{i=1}^{N_{\delta}}$, where $B_i = B(y_i, \frac{1}{3}\delta)$ and $N_{\delta} \lesssim \delta^{-(s+\varepsilon)} $. The constant $\frac{1}{3}$  was chosen so that for any $z \in B(y_i, \frac{1}{3}\delta)$,  
\[
B(y_i, \frac{1}{3}\delta) \subseteq T^{\delta}_{e}(z).
\]  
By the pigeonhole principle, there exists a ball $B_{i_0} = B(y_{i_0}, \frac{1}{3}\delta)$ such that at least $\frac{N}{N_{\delta}}$ midpoints $a_j$ are contained in $B_{i_0}$, and the point $y_{i_0}$ is contained in at least $\frac{N}{N_{\delta}}$ tubes. Assume the tubes are labelled such that $y=y_{i_0}$ belongs to the first $M_{\delta}$ tubes, i.e., $y \in T_j$ for $j=1, \dots, M_{\delta}$, where $M_{\delta}$ is the greatest integer less than $\frac{N}{N_\delta}$, see Figure \ref{onebush}.
	\begin{figure}[h]
		\centering
		\includegraphics[scale=0.55]{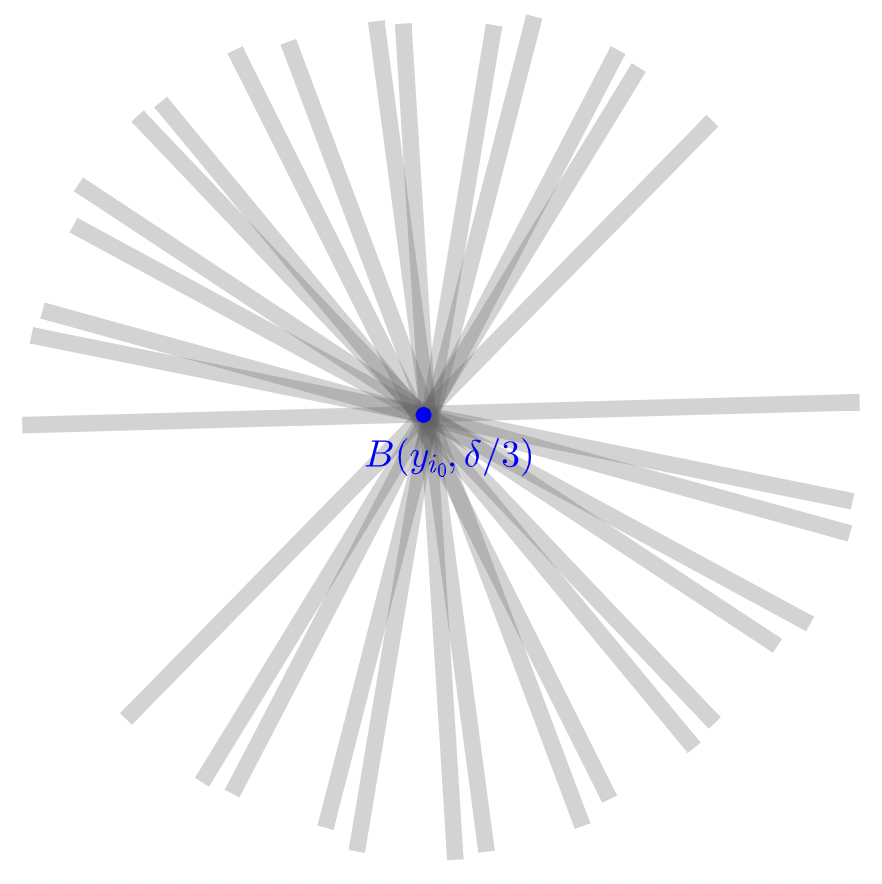} 
		\caption{$M_{\delta}$ tubes whose centres are in the ball $B(y_{i_0},\frac{1}{3}\delta)$.}
		\label{onebush}
	\end{figure}

	By geometric considerations, there exists a constant $c$ depending only on $n$ such that
	\mm{
	\abso{B(y,c \lambda) \cap T^{\delta}_e (a) } \leq \frac{\lambda}{2} |T^{\delta}_e(a)|
	}
	for any $e \in S^{n-1}$ and $a \in \mathbb{R}^{n}$. From $\eqref{thm1.2_1}$ for $j=1, \dots, M_{\delta}$, 
	\mm{
	\abso{E \cap T_j \backslash B(y,c \lambda)} \geq \frac{\lambda}{2}|T_j| \approx \lambda \delta^{n-1}.
	}
	By Proposition \ref{propgeo}, there exists another constant $b \geq c$ such that for all $e,e' \in S^{n-1}$ and $a, a' \in \mathbb{R}^n$, 
	\begin{equation}
		\textup{diam}(T^{\delta}_{e}(a) \cap T^{\delta}_{e'}(a')) \leq \frac{b \delta}{|e-e'|}. \label{thm1.2_2}
	\end{equation}
	Let $\{ e_1', \dots e_{m'}'\} $ be a maximal $\frac{b \delta }{c \lambda}$-separated subset of $\{ e_1, \dots e_{M_{\delta}}\}$. Since $\frac{b \delta}{c \lambda }> \delta$, the balls $B(e_k', \frac{2b\delta}{c \lambda})$, $k=1,\dots, m'$, cover the disjoint balls $B(e_j, \frac{\delta}{3})$ for $j=1, \dots ,M_{\delta}$.
	Thus, 
	\mm{
	\frac{N}{N_{\delta}} \delta^{n-1} \approx \abso{\bigcupp{j=1}{M_{\delta}}B(e_j, \frac{\delta}{3})}_{n-1} \leq \abso{\bigcupp{k=1}{m'}B(e_k',\frac{2b\delta}{c \lambda})}_{n-1} \lesssim m' (\frac{\delta}{\lambda})^{n-1},
	}
which implies  
\[
m' \gtrsim \lambda^{n-1} \frac{N}{N_{\delta}}.
\]  
	It follows from $\eqref{thm1.2_2}$ that the sets $\paree{E \cap T'_k \backslash B(y,c \lambda)}_{k=1}^{m'}$ are disjoint because	
	\mm{
	\textup{diam}(T_k' \cap T_s') \leq \frac{b \delta}{|e_k'-e_s'|} \leq \frac{b\delta }{b \delta / (c \lambda)} = c \lambda
	}
	for any $k \neq s$ in $1, \dots,m'$.
	Therefore, 
	\begin{equation}
		|E|_n \gtrsim \lambda \delta^{n-1}m' \gtrsim \lambda \delta^{n-1} \lambda^{n-1} \frac{N}{N_{\delta}} \gtrsim N \lambda^n \delta^{n+s+\varepsilon -1}. \nonumber
	\end{equation}
	Together with \eqref{sec3.2_1}, we obtain \eqref{equal2.2}.

\subsection{Proof of Theorem \ref{mainresult}} \label{bushproof}
By checking a simple example for $f$ in \eqref{result_n-1}, we obtain a simple relationship between $p_{n-1}$ and $h_{n-1}$.
\begin{prop}
	Suppose for some $h_{n-1} >0$ and $p_{n-1} >1$, 
	\begin{equation}
		\nnn{(f)^*_{\delta}}{L^{p_{n-1}}(S^{n-2})}{} \lesssim_{\varepsilon}  \delta^{-h_{n-1}-\varepsilon}\nnn{f}{L^{p_{n-1}}(\mathbb{R}^{n-1})}{}\nonumber 
	\end{equation}
	holds for all $f \in L^1_{loc}(\RR^{n-1})$ and $0 < \delta <1$ and some  $\varepsilon >0$. Then 
	\begin{equation}
		n-1 \leq (1+h_{n-1})p_{n-1}+ \varepsilon p_{n-1}. \label{handp}
	\end{equation}
\end{prop}

\begin{proof}
	Let $f=\chi_{B(0, \delta)}$. We first know that 
	\[
	\nnn{f}{L^{p_{n-1}}(\mathbb{R}^{n-1})}{}=(\delta^{n-1})^{1/p_{n-1}}.
	\]
	Since for all $e \in S^{n-2}$, the $\delta$-tube $T^{\delta}_e(0)$ contains the ball $B(0, \delta)$, the Kakeya maximal function satisfies
	\[
	(f)^*_{\delta}(e) \approx \frac{\delta^{n-1}}{\delta^{n-2}}=\delta.
	\]
	Then $p_{n-1}$ and $h_{n-1}$ satisfy
	\[
	\delta \lesssim_{\varepsilon} \delta^{-h_{n-1}-\varepsilon+\frac{n-1}{p_{n-1}}}
	\]
	for all $0 < \delta <1$ and $\varepsilon>0$. Therefore, the exponents satisfy
	\[
	1 \geq -h_{n-1}-\varepsilon+\frac{n-1}{p_{n-1}},
	\]
	which, upon rearranging, gives  \eqref{handp}.
\end{proof}

Note that we can rewrite the conclusion of Theorem \ref{mainresult} as  
\begin{equation}
	\lambda \abso{\paree{e \in S^{n-1}: \cK_{\delta,A}(\chi_{E})(e)> \lambda}}^{\frac{1}{p}}_{n-1} \lesssim_{\varepsilon} \delta^{-\beta-\varepsilon}\abso{E}^{\frac{1}{p}}_n \label{mainresultformula_2}
\end{equation}
for all measurable sets $E \subseteq \RR^n$, $0< \delta<1$, $0< \lambda<1$, $\varepsilon >0$ and all $A \subseteq \RR^n$ with $\ubd A \leq s$. We split the proof into two cases. Let 
\[
E_{\lambda}=\paree{e \in S^{n-1}: \cK_{\delta,A}(\chi_{E})(e)> \lambda}.
\]
We assume that $E_{\lambda}$ is nonempty.  

\textbf{Case 1}: when $\lambda \leq \delta$.  
For any $\xi \in E_{\lambda}$, there exists a $\delta$-tube $\tube{\xi}{a_{\xi}}$ with $a_{\xi} \in A$ such that 
\begin{equation}
	\abso{E \cap \tube{\xi}{a_{\xi}}}_{n} \gtrsim \lambda \delta^{n-1}. \nonumber
\end{equation}
Therefore, 
\[
\abso{E}_n \geq \abso{E \cap \tube{\xi}{a_{\xi}}}_{n} \gtrsim \lambda \delta^{n-1}.
\]
Comparing with \eqref{mainresultformula_2}, it suffices to show 
\[
\lambda^{p-1}\delta^{\beta p + \varepsilon p -(n-1)} \abso{E_{\lambda}}_{n-1} \lesssim_{\varepsilon}1.
\]
Since the set $E_{\lambda} \subseteq S^{n-1}$ has finite measure and $\lambda \leq \delta$, it suffices to show 
\begin{equation}
	\delta^{p-1+\beta p + \varepsilon p -(n-1) } \lesssim_{\varepsilon}1. \label{beforeexp}
\end{equation}
Recalling 
\[
p=\frac{p_{n-1}+n(p_{n-1}-1)+1}{p_{n-1}}, \quad \beta=\frac{h_{n-1}p_{n-1}+sp_{n-1}-s}{p_{n-1}+n(p_{n-1}-1)+1}
\]
and substituting $p$ and $\beta$ into \eqref{beforeexp}, we finally need to prove the following inequality for the exponent in \eqref{beforeexp}:
\begin{equation}
	\frac{(1+h_{n-1})p_{n-1}-(n-1)+sp_{n-1}-s}{p_{n-1}}+ \varepsilon p \geq0. \nonumber
\end{equation}
By relation \eqref{handp}, this inequality holds.

\textbf{Case 2}: when $\lambda > \delta$.
We first define the notion of a bush.  
\begin{definition}
    Suppose $\mathcal{T}$ is a finite collection of $\delta$-tubes $\paree{T_1, \dots , T_N}$. If there exists a point $x$ contained in every tube $T_i \in \mathcal{T}$ for $i=1, \dots, N$, we call the union  
    \[
    \BB=\bigcup_{i=1}^{N} T_i
    \]  
    a \textit{bush}.  
\end{definition}  
For example, Figure \ref{onebush} illustrates a bush consisting of $M_{\delta}$ tubes, where the point $y_{i_0}$ is contained in each tube.\\
Let $E = E_0$ be a measurable subset of $\mathbb{R}^n$, and define  
\[
D_0 = \{ e \in S^{n-1} : \McK{\chi_E}(e) > \lambda \}.
\]  
We can assume $D_0$ and $E_0$ are both non-empty, and the measure of $D_0$ is $$|D_0|_{n-1} =: \epsilon_0.$$  
Let $\EE_0$ be a $\frac{10 \delta}{\lambda}$-separated subset of $D_0$ with cardinality  
\[
\# \EE_0 \gtrsim \epsilon_0 \left( \frac{\lambda}{\delta} \right)^{n-1}.
\]  
For each $\xi \in \EE_0$, there exists a point $a_\xi \in A$ such that we can find a tube $\tube{\xi}{a_{\xi}}$ satisfying  
\begin{equation}
\abso{E \cap \tube{\xi}{a_\xi}} \gtrsim \lambda \delta^{n-1}. \label{1111111}
\end{equation} 
By the assumption $\ubox A \leq s$, we know that for the same $\varepsilon>0$ in \eqref{mainresultformula_2}, there exists an $N_\delta \lesssim \delta^{-s-\varepsilon}$ such that $A$ can be covered using at most $N_\delta$ balls of radius $\delta$.  Since there are at least $\epsilon_0 \left(\frac{\lambda}{\delta} \right)^{n-1}$ tubes whose centres are in $A$, the pigeonhole principle implies that at least one of these $\delta$-balls must contain at least  
\[
\frac{\epsilon_0 \left(\frac{\lambda}{\delta} \right)^{n-1}}{\delta^{-s-\varepsilon}} = \epsilon_0 \lambda^{n-1} \delta^{s+\varepsilon -(n-1)}
\]  
tube centres.  

Thus, we can choose the directions of these tubes to form a set $\FF_0 \subseteq S^{n-1}$ with  
\begin{equation}
\# \FF_0 \gtrsim \epsilon_0 \lambda^{n-1} \delta^{s+\varepsilon -(n-1)}, \label{2222222}
\end{equation}
and construct a bush $\BB_0$ using these directions:  
\[
\BB_0 = \bigcup_{\xi \in \FF_0} \tube{\xi}{a_\xi}.
\]  
By the construction of the bush $\BB_0$, there exists a point $x_0$ contained in each tube of $\BB_0$.  

Using the assumption $\lambda > \delta$, for any tube $\tube{\xi}{a_\xi}$ containing $x_0$, we have  
\[
\abso{\tube{\xi}{a_\xi} \cap B(x_0, \frac{\lambda}{3})}_n \lesssim \frac{2\lambda}{3} \delta^{n-1}.
\]
Using this and \eqref{1111111},  
\begin{equation}
	\abso{E \cap \pare{\tube{\xi}{a_\xi} \setminus B(x_0, \frac{\lambda}{3})}}_n \gtrsim \frac{\lambda}{3} \delta^{n-1}. \label{first}
\end{equation}

According to Proposition 2.1, for any two distinct elements $\xi_1, \xi_2 \in \FF_0$, we have the following estimate for the diameter of the intersection of the tubes:
\[
\text{diam} \pare{\tube{\xi_1}{a_{\xi_1}} \cap \tube{\xi_2}{a_{\xi_2}}} \lesssim \frac{\delta}{\theta(\xi_1,\xi_2)} < \frac{\delta}{10 \delta / \lambda} = \frac{\lambda}{10}.
\]
Since $x_0$ is contained in both $\tube{\xi_1}{a_{\xi_1}}$ and $\tube{\xi_2}{a_{\xi_2}}$, it follows that  
\[
\tube{\xi_1}{a_{\xi_1}} \cap \tube{\xi_2}{a_{\xi_2}} \subseteq B(x_0, \frac{\lambda}{3}).
\]
This shows that the sets  
\[
\left\{ E \cap \pare{\tube{\xi}{a_\xi} \setminus B(x_0, \frac{\lambda}{3})} \right\}_{\xi \in \FF_0}
\]  
are disjoint. Summing \eqref{first} over all $\xi \in \FF_0$, we obtain  
\begin{align}
\abso{E \cap \BB_0}_n & \geq \abso{E \cap (\BB_0 \setminus B(x_0, \frac{\lambda}{3})) } _n\nonumber\\ 
& = \sum_{\xi \in \FF_0} \abso{E \cap (\tube{\xi}{a_{\xi}} \setminus B(x_0, \frac{\lambda}{3}))}_n \nonumber\\
& \gtrsim  \# \FF_0 \cdot \frac{\lambda}{3} \delta^{n-1} \nonumber\\
& \gtrsim \frac{\lambda}{3} |\BB_0|_n. \nonumber
\end{align}
We can express the above estimate as  
\begin{equation}
	\frac{1}{\lambda} \abso{E \cap \BB_0}_n \gtrsim |\BB_0|_n. \label{bush_0}
\end{equation}

The construction begins with a subset $E = E_0 \subseteq \RR^n$, from which we obtain a bush $\BB_0$. The next step is to apply the same construction to the remaining set $E_1 = E \setminus \BB_0$, provided that the level set  
\[
D_1 = \left\{ e \in S^{n-1} : \McK{\chi_{E_1}}(e) > \frac{\lambda}{2} \right\}
\]  
satisfies $\abso{D_1}_{n-1} \geq \frac{1}{4} \epsilon_0$. 

Similarly, we choose a $\frac{10 \delta}{\lambda}$-separated subset $\EE_1$ of $D_1$. Since $\abso{D_1}_{n-1} \geq \frac{1}{4} \epsilon_0$, we can again select at least $\epsilon_0 (\frac{\lambda}{\delta})^{n-1}$ tubes with directions in $\EE_1$ and centres in $A$. 

By the pigeonhole principle and the assumption $\ubox A \leq s$, we observe that at least $\epsilon_0 \lambda^{n-1} \delta^{s+\varepsilon -(n-1)}$ of these tubes have centres within the same $\delta$-ball. We then collect the directions of these tubes into the set $\FF_1$, forming a new bush $\BB_1$ that satisfies  
\begin{equation}
	\frac{1}{\lambda} \abso{E_1 \cap \BB_1}_n \gtrsim |\BB_1|_n. \label{bush_1}
\end{equation}

We continue this process by defining 
\begin{equation}
	D_i = \left\{ e \in S^{n-1} : \McK{\chi_{E_i}}(e) > \frac{\lambda}{2} \right\}, \label{defdi}
\end{equation} 
and constructing bushes $\BB_{i}$ iteratively, until the level set at the next step, $D_m$, satisfies the stopping condition  
\begin{equation}
	\abso{D_m}_{n-1} < \frac{1}{4} \epsilon_0. \label{stopcondition}
\end{equation}

The construction at each step and the stopping condition are illustrated in Figure \ref{googgraph}.

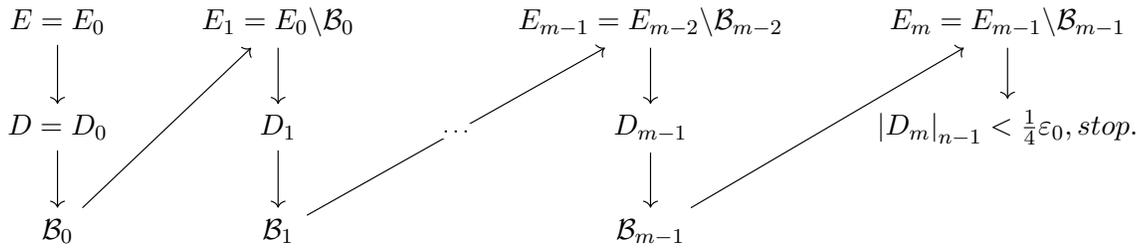
\begin{figure*}[h]

\centering

	\begin{tikzcd}[scale=2.]
E=E_0 \arrow[d]   & E_1=E_0 \backslash \BB_0 \arrow[d]  &                & E_{m-1}=E_{m-2} \backslash \BB_{m-2} \arrow[d]   & E_m=E_{m-1} \backslash \BB_{m-1} \arrow[d] \\
D=D_0 \arrow[d]   & D_1 \arrow[d]  &  & D_{m-1} \arrow[d]   & \abso{D_m}_{n-1} < \frac{1}{4} \epsilon_0, stop.        \\
\BB_0 \arrow[ruu] & \BB_1 \arrow[rruu, " \ . \, . \, .\ " description] &                & \BB_{m-1} \arrow[ruu] &     
\end{tikzcd}
\caption{Construction of the bushes $\BB_i$.}
\label{googgraph}
\end{figure*}

The next step is to show that this construction must terminate after a finite number of steps.

\begin{prop} \label{stoplemma}
	The construction process terminates before step $m$, where  
	\begin{equation}
		m \leq \frac{1}{\epsilon_0} \abso{E}_n \delta^{-s-\varepsilon} \lambda^{-n}.\label{stopformula}
	\end{equation}
\end{prop}

\begin{proof}
	By construction, each set is defined recursively as $E_{i+1} = E_i \setminus \BB_i$, and the sets $\{E_i \cap \BB_i\}_{i=0}^{m-1}$ are disjoint. From the same reasoning as in \eqref{bush_0} or \eqref{bush_1}, we obtain  
	\begin{equation}
		\frac{\abso{E}_n}{\lambda} \geq \sum_{i=0}^{m-1} \frac{\abso{E_i \cap \BB_i}_n}{\lambda} \gtrsim \sum_{i=0}^{m-1} \abso{\BB_i}_n. \label{bush_3}
	\end{equation}

	Each bush $\BB_i$ consists of $\# \FF_i$ many $\delta$-tubes, whose directions are $\frac{10 \delta}{\lambda}$-separated. From basic geometric considerations and \eqref{2222222}, we estimate  
	\begin{equation}
		\abso{\BB_i}_n \approx \# \FF_i \delta^{n-1} \gtrsim \epsilon_0 \lambda^{n-1} \delta^{s+\varepsilon}.\nonumber
	\end{equation}
	Substituting this into \eqref{bush_3}, we obtain  
	\begin{equation}
		\frac{\abso{E}_n}{\lambda} \gtrsim \sum_{i=0}^{m-1} \abso{\BB_i}_n \gtrsim m \epsilon_0 \lambda^{n-1} \delta^{s+\varepsilon}.\label{333333}
	\end{equation}
	Rearranging gives the desired bound \eqref{stopformula}.  
\end{proof}

The final step of the above construction yields
\mm{E_m = E \cap \pare{\bigcup_{i=0}^{m-1} B_i}^c,}
where $|D_m|_{n-1} < \frac{1}{4} \epsilon_0$. Define 
\mm{\overline{E} = E \setminus E_m = E \cap \pare{\bigcup_{i=0}^{m-1} B_i},}
and let
\begin{equation}
	\overline{D} = \{ e \in S^{n-1} : \McK{\chi_{\overline{E}}}(e) > \frac{\lambda}{2} \}. \label{dbardef}
\end{equation}
We will show that 
\mm{D_0 \subseteq D_m \cup \overline{D}.}

Indeed, for any $e \in D_0$, there exists a $\delta$-tube $\tube{e}{a_e}$ centred at some $a_e \in A$ such that 
\mm{\frac{1}{\abso{\tube{e}{a_e}}_n} \abso{E \cap \tube{e}{a_e}}_n > \lambda.}
Since $E_m$ and $\overline{E}$ form a partition of $E$, we obtain
\mm{\frac{1}{\abso{\tube{e}{a_e}}_n} \abso{E_m \cap \tube{e}{a_e}}_n + \frac{1}{\abso{\tube{e}{a_e}}_n} \abso{\overline{E} \cap \tube{e}{a_e}}_n > \lambda.}
Thus, at least one of the terms on the left-hand side must be greater than $\frac{\lambda}{2}$, implying $e \in D_m \cup \overline{D}$ from the definition \eqref{defdi} and \eqref{dbardef}.   Since  
\mm{\abso{D_m}_{n-1} + \abso{\overline{D}}_{n-1} \geq \abso{D_0}_{n-1} = \epsilon_0}
and the stopping condition ensures
\mm{\abso{D_m}_{n-1} < \frac{1}{4} \epsilon_0,}
it follows that
\begin{equation}
	\abso{\overline{D}}_{n-1} \geq \frac{1}{4} \epsilon_0. \label{measuredbar}
\end{equation}
For any $\xi \in \overline{D}$, there exists a $\delta$-tube $\tube{\xi}{a_\xi}$ centred at $a_\xi \in A$ such that  
\mm{\frac{1}{\abso{\tube{\xi}{a_\xi}}_n} \abso{\bigcup_{i=0}^{m-1} (E \cap \BB_i) \cap \tube{\xi}{a_\xi}}_n > \frac{\lambda}{2}.}
This implies
\begin{equation}
\begin{split}
	\frac{\summ{i=0}{m-1}\abso{\BB_i \cap \tube{\xi}{a_{\xi}}}_n}{\abso{\tube{\xi}{a_{\xi}}}_n} & \geq \frac{\summ{i=0}{m-1}\abso{(E \cap \BB_i) \cap \tube{\xi}{a_{\xi}}}_n}{\abso{\tube{\xi}{a_{\xi}}}_n}\\
	& \geq \frac{\abso{\bigcup_{i=0}^{m-1} (E \cap \BB_i) \cap \tube{\xi}{a_\xi}}_n}{\abso{\tube{\xi}{a_{\xi}}}_n}
	> \frac{\lambda}{2} \label{beforebourgainlemma}
\end{split}
\end{equation}

The final part of the proof introduces the maximal function associated with neighbourhoods of parallelograms.

\begin{defn}
	For any $e \in S^{n-1}$ and any two distinct points $x_1, x_2 \in \RR^n$, let $P_{e}(x_1, x_2)$ be a parallelogram formed by two parallel edges $I_e(x_1)$ and $I_e(x_2)$. Let $\TT_{e}(x_1, x_2)$ denote the $\delta$-neighbourhood of $P_e(x_1, x_2)$. We define the maximal function over all such neighbourhoods of parallelograms as
	\begin{equation}
		\cM_{\delta}(f)(e) = \sup_{x_1, x_2} \frac{1}{\abso{\TT_e(x_1, x_2)}_n} \int_{\TT_e(x_1, x_2)} f(x) \,dx.\nonumber
	\end{equation}
\end{defn}

We now state a lemma from \cite{B91}. In fact, only the case $n=3$ is explicit in  \cite[Lemma 1.52]{B91} but the extension to $\RR^n$  is implicit (and used) in \cite[Page 158, (2.8)]{B91}.

\begin{lma} \label{bourgainlemma}
	Suppose that for all $\varepsilon>0$,
	\begin{equation}
		\nnn{(f)^*_{\delta}}{L^{p_{n-1}}(S^{n-2})}{} \lesssim_{\varepsilon}  \delta^{-h_{n-1}-\varepsilon}\nnn{f}{L^{p_{n-1}}(\mathbb{R}^{n-1})}{} \nonumber
	\end{equation}
	holds for some $h_{n-1} > 0$ and $p_{n-1} \geq 1$. Then for all $\varepsilon > 0$, the maximal function over $\delta$-neighbourhoods of parallelograms in $\mathbb{R}^n$ satisfies  
	\begin{equation}
		\nn{\cM_\delta (f)}_{L^{p_{n-1}}(S^{n-1})} \lesssim_{\varepsilon} r^{-\frac{1}{p_{n-1}}} \delta^{-h_{n-1}-\varepsilon} \nn{f}_{L^{p_{n-1}}(\mathbb{R}^{n})} \label{bourgainlemmaformula}
	\end{equation}
	for all function $ f \in L^1_{\text{loc}}(\mathbb{R}^n) $ supported in $ B(0,2r) \setminus B(0,r) $.
\end{lma}
For each $i=0, \dots, m-1$, let $x_i$ be the centre of the bush $\BB_i$. From geometric considerations (see Figure \ref{parapip}), if we move the blue tube from $\tube{\xi}{a_{\xi}}$ to the red tube $\tube{\xi}{a_0}$ following the arrows, the proportion of the intersection between the moving tube and the bush $\BB_i$ remains comparable to the proportion of the intersection $\BB_i \cap \TT_\xi (x_i, a_\xi)$ within $\TT_\xi (x_i, a_\xi)$. Therefore, we obtain
\begin{equation}
	\frac{1}{\abso{\tube{\xi}{a_\xi}}_n} \abso{\BB_i \cap \tube{\xi}{a_\xi}}_n
	\lesssim \frac{1}{\abso{\TT_\xi (x_i, a_\xi)}_n} \abso{\BB_i \cap \TT_\xi (x_i, a_\xi)}_n \label{bush_graph}
\end{equation}
for any $a_{\xi} \in \RR^n$.
\begin{figure}[htb]
		\centering
		\includegraphics[scale=0.45]{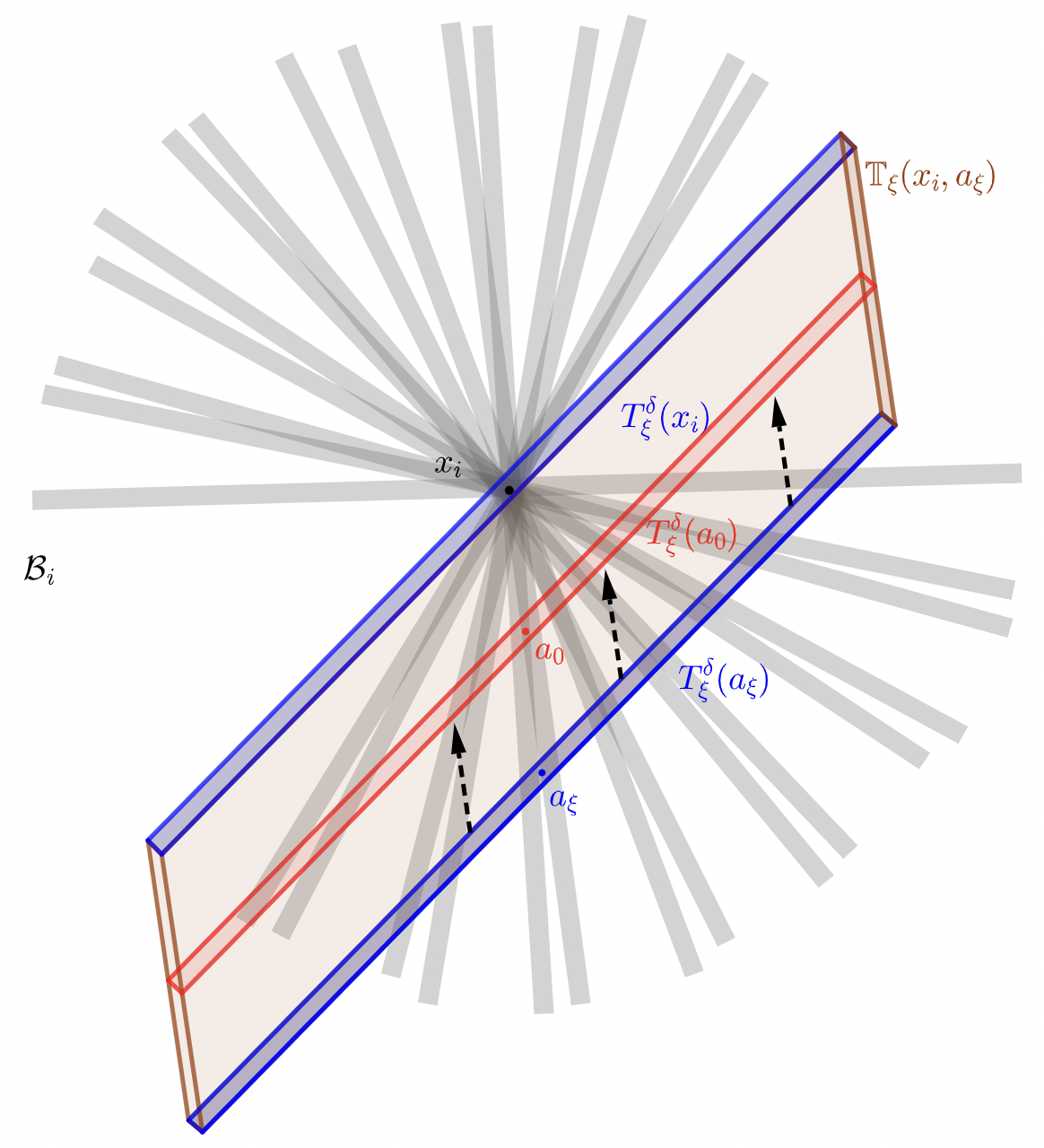} 
		\caption{Geometric Observation for \eqref{bush_graph}. $a_0$ is the midpoint of $x_i$ and $a_{\xi}$.}
		\label{parapip}
\end{figure}

Using \eqref{beforebourgainlemma}, we estimate
\begin{equation}
	\begin{aligned}
		\lambda &\lesssim \sum_{i=0}^{m-1} \frac{1}{\abso{\tube{\xi}{a_\xi}}_n} \abso{\BB_i \cap \tube{\xi}{a_\xi}}_n \\
		& \lesssim \sum_{i=0}^{m-1} \frac{1}{\abso{\TT_\xi (x_i, a_\xi)}_n} \abso{\BB_i \cap \TT_\xi (x_i, a_\xi)}_n \\
		& \leq \sum_{i=0}^{m-1} \cM_\delta (\chi_{\BB_i})(\xi) \\
		& \lesssim \sum_{i=0}^{m-1} \sum_{k=0}^{\log\frac{1}{\delta}} \cM_{\delta}(\chi_{\BB^k_i})(\xi), \label{bush_5}
	\end{aligned}
\end{equation}
where $\BB^k_i = (\BB_i - x_i) \cap (B(0,2^{k+1}\delta) \setminus B(0,2^k\delta))$.
Taking the exponent $p_{n-1}$ on both sides of \eqref{bush_5} and applying Jensen's inequality, we obtain
\begin{equation}
	\lambda^{p_{n-1}}\lesssim\pare{ \sum_{i=0}^{m-1} \sum_{k=0}^{\log\frac{1}{\delta}} \cM_{\delta}(\chi_{\BB^k_i})(\xi)}^{p_{n-1}} \leq \pare{m\log\frac{1}{\delta}}^{p_{n-1}-1}\sum_{i=0}^{m-1} \sum_{k=0}^{\log\frac{1}{\delta}} \pare{\cM_{\delta}(\chi_{\BB^k_i})(\xi)}^{p_{n-1}}.\label{highbush_1}
\end{equation}
Integrating \eqref{highbush_1} over $\overline{D}$, and by \eqref{measuredbar}, it follows that 
\begin{equation}
\begin{aligned}
	\epsilon_0 \lambda^{p_{n-1}} &\lesssim \pare{m \log \frac{1}{\delta}}^{p_{n-1}-1} \sum_{i=0}^{m-1} \sum_{k=0}^{\log\frac{1}{\delta}} \int_{\overline{D}} \pare{\cM_{\delta}(\chi_{\BB^k_i})(\xi)}^{p_{n-1}} d\xi\\
	& = \pare{m \log \frac{1}{\delta}}^{p_{n-1}-1} \sum_{i=0}^{m-1} \sum_{k=0}^{\log\frac{1}{\delta}} \nn{\cM_{\delta}(\chi_{\BB^k_i})}^{p_{n-1}}_{L^{p_{n-1}}(S^{n-1})}.  \label{bush_7}
\end{aligned}
\end{equation}
By Lemma \ref{bourgainlemma},
\begin{equation}
	 \epsilon_0 \lambda^{p_{n-1}} \lesssim_{\varepsilon} \pare{m \log \frac{1}{\delta}}^{p_{n-1}-1} \sum_{i=0}^{m-1} \sum_{k=0}^{\log\frac{1}{\delta}} (2^k\delta)^{-1} \delta^{-p_{n-1}h_{n-1}-\varepsilon} \abso{\BB^k_i}_n. \nonumber
\end{equation}
The geometric observation, see Figure \ref{partbushgraph}, demonstrates that
\begin{equation}
	\abso{\BB^k_i}_n \lesssim 2^k \delta \abso{\BB_i}_n. \label{partbushformula}
\end{equation}
From the graph, we can see that the brown tube intersects the blue annulus, and the portion of the brown tube inside the annulus has a length of $2^k \delta$. Additionally, multiple tubes overlap within the annulus, contributing to the estimate in equation \eqref{partbushformula}.
\begin{figure}[htb]
		\centering
		\includegraphics[scale=0.45]{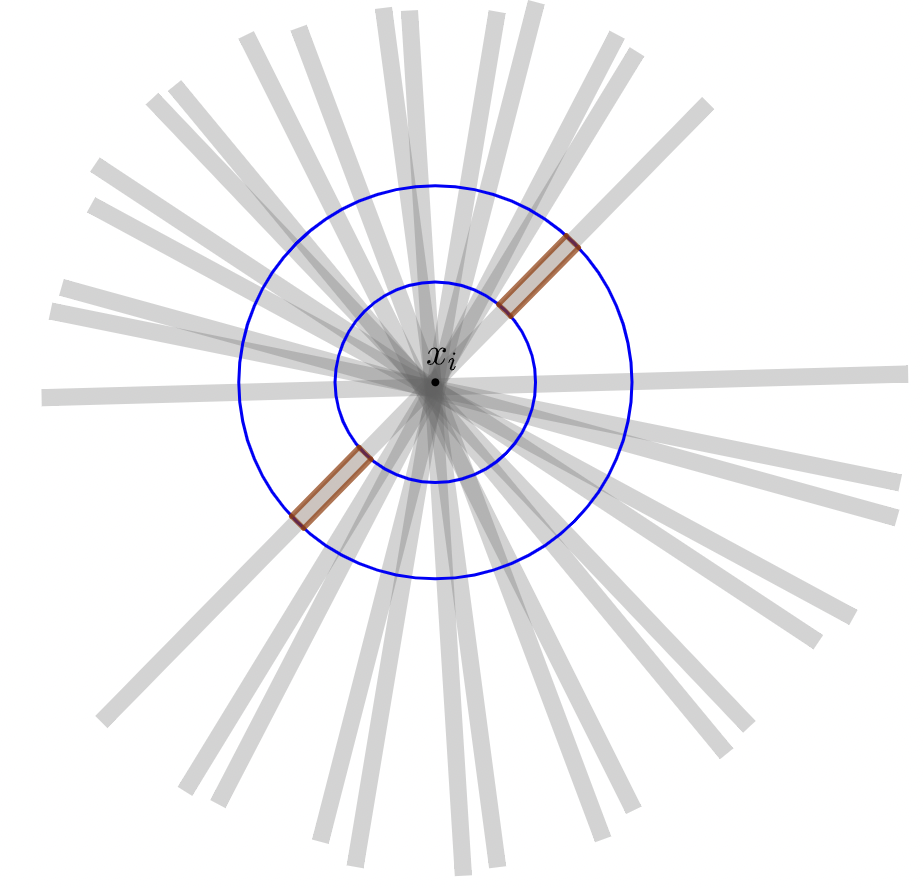} 
		\caption{Geometric Observation for \eqref{partbushformula}.}
		\label{partbushgraph}
\end{figure}

Substituting the stop condition in Proposition \ref{stoplemma} for $m$ and the first part of \eqref{333333} for $\summ{i=0}{m-1}\abso{\BB_i}_n$,  
\begin{equation}
\begin{aligned}
	\epsilon_0 \lambda^{p_{n-1}} &\lesssim_{\varepsilon} \pare{m \log \frac{1}{\delta}}^{p_{n-1}-1}  \sum_{k=0}^{\log\frac{1}{\delta}} (2^k\delta)^{-1} \delta^{-p_{n-1}h_{n-1}-\varepsilon} 2^k \delta \sum_{i=0}^{m-1}\abso{\BB_i}_n\\
	& \leq m^{p_{n-1}-1} \pare{\log \frac{1}{\delta}}^{p_{n-1}} \delta^{-p_{n-1}h_{n-1}-\varepsilon} \frac{\abso{E}_n}{\lambda}\\
	& \lesssim_{\varepsilon} \pare{\frac{1}{\epsilon_0} \abso{E}_n \delta^{-s-\varepsilon} \lambda^{-n}}^{p_{n-1}-1}\delta^{-p_{n-1}h_{n-1}-2\varepsilon} \frac{\abso{E}_n}{\lambda}. \nonumber
\end{aligned}
\end{equation}
Then it follows that  
\begin{equation}
	\lambda \epsilon_0^{\frac{p_{n-1}}{p_{n-1}+n(p_{n-1}-1)+1}} \lesssim_{\varepsilon} \delta^{-\frac{p_{n-1}h_{n-1}+(s+\varepsilon)(p_{n-1}-1)+2\varepsilon}{p_{n-1}+n(p_{n-1}-1)+1}}\abso{E}_n^{\frac{p_{n-1}}{p_{n-1}+n(p_{n-1}-1)+1}},\nonumber
\end{equation}
which is the desired \eqref{mainresultformula_2}, albeit with a different $\varepsilon$.
\section*{Acknowledgements}

We thank Tam\'as Keleti for telling us the proof of Proposition A (2) .

\end{document}